\documentclass{article}
  
\usepackage[utf8]{inputenc}
\usepackage[T1]{fontenc}
\usepackage{geometry}
\usepackage[fleqn]{amsmath}
\usepackage{amsfonts}
\usepackage{xcolor}
\usepackage[english]{babel}
\usepackage{amsthm}
\usepackage{tikz}
\title{The Axiomatization of Affine Oriented Matroids Reassessed}
\date{7 Jul 2016} 
\author{Andrea Baum and Yida Zhu}

\theoremstyle{definition}
\newtheorem{Def}{Definition}[section]
\newtheorem{Bem}{Remark}[section]
\newtheorem{Bsp}{Example}[section]

\theoremstyle{plain}
\newtheorem{Kor}{Corollary}[section]
\newtheorem{Lem}{Lemma}[section]
\newtheorem{Prop}{Proposition}[section]
\newtheorem{Thm}{Theorem}[section]

\DeclareMathOperator{\sym}{Sym}
\DeclareMathOperator{\asym}{Asym}

\begin{document}
\maketitle

\begin{abstract}
In an unpublished manuscript of 1992, Johan Karlander has given an axiomatization of affine oriented matroids, which can be thought of as oriented matroids with a hyperplane at infinity. A closer examination of the text revealed an invalid construction and an incorrect argument in the proof of his main theorem.
This paper provides an alternative argument to fix and slightly simplify the proof of the main theorem.
\end{abstract}

\section{Basic Notions and Motivation}
Oriented matroids can be thought of as combinatorial abstractions of real hyperplane arrangements,
which arise as fundamental objects in various mathematical theories: they arise from
inequality systems in linear programming, from facets of convex polytopes and so on. Real
hyperplane arrangements have also been studied in discrete geometry with respect to their combinatorial
structure, that is, how they partition space.
\vspace{0.5cm}
\\
\emph{\textbf{Arrangements of Hyperplanes}}\vspace{0.5cm}

A finite family $\mathcal{H} = \{H_e:e\in E\}$ of affine hyperplanes in $\mathbf{R}^{d}$ is called an \emph{arrangement of hyperplanes}. For the sake of simplicity, we always make the \emph{regularity assumption}: there is a subset of $\mathcal{H}$ whose intersection is a single point. However, as we will see in the next section, any arrangement of hyperplanes can be considered as the intersection of a central arrangement (of which the intersection is a single point) with another hyperplane.

Associated with each $H_e$ in the arrangement, there are two open halfspaces bounded by $H_e$, which we call the \emph{positive side} (\emph{plus side}) and \emph{negative side} (\emph{minus side}) of $H_e$, denoted by ${H_e}^{+}$ and ${H_e}^{-}$.
It does not matter which side is the plus side, it is only important that the assignment of $+$ and $-$ is fixed.
Then for every vector $x \in \mathbf{R}^{d}$ we can define a \emph{sign vector} $\sigma(x) = (\sigma_1(x),\sigma_2(x),...,\sigma_n(x))$ where $\sigma_e(x) = +$ if $x \in {H_e}^{+}$, $\sigma_e(x) = -$ if $x \in {H_e}^{-}$ and $\sigma_e(x) = 0$ if $x \in H_e$. This sign vector encodes the position of $x$ with respect to each hyperplane.

Figure 1 illustrates an arrangement of five hyperplanes in $\mathbf{R}^{2}$ which satisfies the regularity assumption.
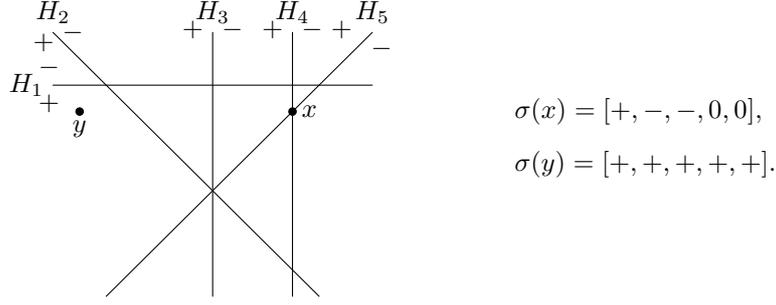
\begin{figure}
\centering
\begin{tikzpicture}[scale=0.7,cap=round,>=latex]
  \coordinate (center) at (0,0);
  \draw (0,-3) -- (0,2);
  \draw (0,2) node[anchor=south]{$H_3$};
  \draw (0,1.7) node[anchor=south east]{$+$};
  \draw (0,1.7) node[anchor=south west]{$-$};
  \draw (1.5,-3) -- (1.5,2);
  \draw (1.5,2) node[anchor=south]{$H_4$};
  \draw (1.5,1.7) node[anchor=south east]{$+$};
  \draw (1.5,1.7) node[anchor=south west]{$-$};
  \draw (-3,1) -- (3,1);
  \draw (-3,1) node[anchor=east]{$H_1$};
  \draw (-2.7,1) node[anchor=south east]{$-$};
  \draw (-2.7,1) node[anchor=north east]{$+$};
  \draw (-3,2) -- (2,-3);
  \draw (-3,2) node[anchor=south]{$H_2$};
  \draw (-2.8,1.8) node[anchor=east]{$+$};
  \draw (-3,2) node[anchor=west]{$-$};
  \draw (3,2) -- (-2,-3);
  \draw (3,2) node[anchor=south]{$H_5$};
  \draw (2.8,1.7) node[anchor=south east]{$+$};
  \draw (2.8,1.7) node[anchor=west]{$-$};

  \filldraw [black] (1.5,0.5) circle (2pt);
  \draw (1.5,0.5) node[anchor=west]{$x$};
  \filldraw [black] (-2.5,0.5) circle (2pt);
  \draw (-2.5,0.5) node[anchor=north]{$y$};

  \draw (5.5,0.5) node[anchor=west]{$\sigma(x) = [+,-,-,0,0],$};
  \draw (5.5,-0.5) node[anchor=west]{$\sigma(y) = [+,+,+,+,+].$};
\end{tikzpicture}
\caption{An arrangement of five hyperplanes.}
\end{figure}

The set of all points $x \in \mathbf{R}^{d}$ having the same sign vector $X =\sigma(x)$ forms a \emph{cell} in the decomposition of $\mathbf{R}^{d}$ induced by $\mathcal{H}$. Let us begin with some basic notions.

\begin{Def}\label{notion}
Let $E$ be a finite set. A \emph{signed subset} (or \emph{sign vector}) $X$ of $E$ is a member of $\{+,-,0\}^E$. We call $E$ the \emph{ground set} of $X$. Every signed subset $X$ can be identified with an ordered pair $(X^{+},X^{-})$ with $X^{+} = \{e \in E : X_e = +\}$, and  $X^{-} = \{e \in E : X_e = -\}$. Let $X,Y$ be signed subsets of $E$, and let $A \subseteq E$ and $\mathcal{W},\mathcal{W'} \subseteq \{+,-,0\}^{E}$. Then\vspace{0.2cm}\\
$\bullet$  $\underline{X}=X^{+} \cup X^{-}$ is the \emph{support} of X and $X^{0} = E - \underline{X}$ is the \emph{zero set} of $X$. If $\mathcal{W}$ has the property that every member of  $\mathcal{W}$ has the same support, then the support of $\mathcal{W}$, denoted by $\underline{\mathcal{W}}$, is this common support. The zero set of $\mathcal{W}$ is defined in the same way.\vspace{0.2cm}\\
$\bullet$	The \emph{composition} $X \circ Y$ of X,Y is defined by:\vspace{0.2cm}

    $\text{ }\text{ }\text{ } (X \circ Y)_e = \begin{cases}
    X_e, & if X_e \neq 0, 	\\
    Y_e, & otherwise.
    \end{cases}$\vspace{0.2cm}\\
The composition of $\mathcal{W},\mathcal{W'}$, denoted by $\mathcal{W} \circ \mathcal{W'}$, is the set of compositions of members from $\mathcal{W}$ and $\mathcal{W'}$. If $\mathcal{W} = \{X\}$, we write $X \circ \mathcal{W'}$ instead of $\{X\} \circ \mathcal{W'}$.\vspace{0.2cm}\\
$\bullet$  The \emph{opposite} $-X$ is the signed subset $(X^{-},X^{+})$ and the function mapping $X$ to $-X$ is called \emph{sign reversal}. $\mathcal{W}$ is symmetric if it contains all opposites of its members, i.e. $\mathcal{W} = -\mathcal{W} = \{-X:X\in \mathcal{W}\}$.
The \emph{restriction} of $X$ to $A$, denoted by $X(A)$ or $X|_A$, is the signed subset $(X^{+} \cap A,X^{-} \cap A)$. $\mathcal{W}(A)$ is the restriction of $\mathcal{W}$ to $A$ defined by $\mathcal{W}(A) = \{W(A) : W\in \mathcal{W}\}$. The \emph{reorientation} ${}_{-A}\!X$ of $X$ on A is the signed subset $X(E-A) \circ (-X)$. \vspace{0.2cm}\\
$\bullet$	The \emph{sign order} $\preceq$ is a partial order on $\{+,-,0\}^{E}$, defined by $X \preceq Y$ if and only if $X^{+} \subseteq Y^{+}$ and $X^{-} \subseteq Y^{-}$. In this case we say that $Y$ \emph{conforms} to $X$. The set of members of $\mathcal{W}$ which conform to $X$ is denoted by $\mathcal{W}_{X}$.\vspace{0.2cm}\\
$\bullet$	The \emph{separation set} $S(X,Y)$ of $X,Y$ is defined by:
\[S(X,Y)= (X^{+}\cap Y^{-}) \cup (X^{-} \cap Y^{+}).\]
\end{Def}

\begin{Bsp}
In the context of Figure 1, we have $E = \{1,2,3,4,5\}$. The support of $X$ is $\{1,2,3\}$. Geometrically, the composition $X \circ Y$ corresponds to the cell next to $X$ and lying on the segment between $X$ and $Y$, as shown in Figure 2. The separation set $S(X,Y) = \{2,3\}$ contains the indices of those hyperplanes "separating" $X$ and $Y$, meaning that $X$ and $Y$ belong to different sides of those hyperplanes.
\end{Bsp}
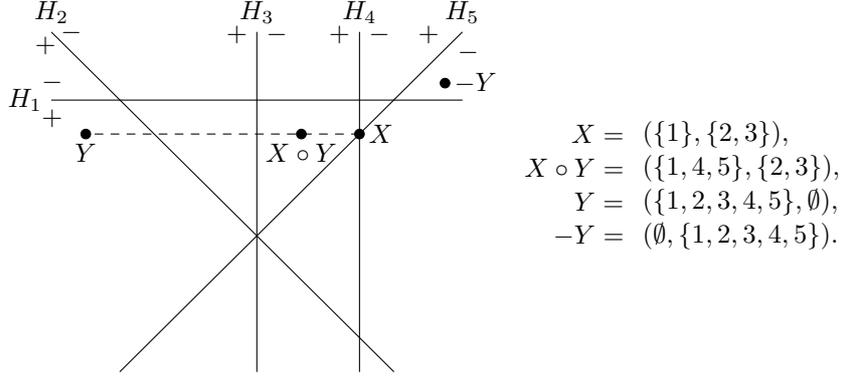
\begin{figure}[h]
\centering
\begin{tikzpicture}[scale=0.9,cap=round,>=latex]
  \coordinate (center) at (0,0);
  \draw (0,-3) -- (0,2);
  \draw (0,2) node[anchor=south]{$H_3$};
  \draw (0,1.7) node[anchor=south east]{$+$};
  \draw (0,1.7) node[anchor=south west]{$-$};
  \draw (1.5,-3) -- (1.5,2);
  \draw (1.5,2) node[anchor=south]{$H_4$};
  \draw (1.5,1.7) node[anchor=south east]{$+$};
  \draw (1.5,1.7) node[anchor=south west]{$-$};
  \draw (-3,1) -- (3,1);
  \draw (-3,1) node[anchor=east]{$H_1$};
  \draw (-2.7,1) node[anchor=south east]{$-$};
  \draw (-2.7,1) node[anchor=north east]{$+$};
  \draw (-3,2) -- (2,-3);
  \draw (-3,2) node[anchor=south]{$H_2$};
  \draw (-2.8,1.8) node[anchor=east]{$+$};
  \draw (-3,2) node[anchor=west]{$-$};
  \draw (3,2) -- (-2,-3);
  \draw (3,2) node[anchor=south]{$H_5$};
  \draw (2.8,1.7) node[anchor=south east]{$+$};
  \draw (2.8,1.7) node[anchor=west]{$-$};

  \filldraw [black] (1.5,0.5) circle (2pt);
  \draw (1.5,0.5) node[anchor=west]{$X$};
  \filldraw [black] (-2.5,0.5) circle (2pt);
  \draw (-2.5,0.5) node[anchor=north]{$Y$};
  \filldraw [black] (2.75,1.25) circle (2pt);
  \draw (2.75,1.25) node[anchor=west]{$-Y$};
  \filldraw [black] (0.65,0.5) circle (2pt);
  \draw (0.65,0.5) node[anchor=north]{$X \circ Y$};

  \draw[dashed] (1.5,0.5) -- (-2.5,0.5);

  \draw (5.5,0.5) node[anchor=west]{$(\{1\},\{2,3\}),$};
  \draw (5.5,0.5) node[anchor=east]{$X = $};
  \draw (5.5,0) node[anchor=west]{$(\{1,4,5\},\{2,3\}),$};
  \draw (5.5,0) node[anchor=east]{$X \circ Y =$};
  \draw (5.5,-0.5) node[anchor=west]{$(\{1,2,3,4,5\},\emptyset),$};
  \draw (5.5,-0.5) node[anchor=east]{$Y = $};
  \draw (5.5,-1) node[anchor=west]{$(\emptyset,\{1,2,3,4,5\}).$};
  \draw (5.5,-1) node[anchor=east]{$-Y =$};  
\end{tikzpicture}

\caption{Geometrical meaning of composition and sign reversion.}
\end{figure}

\begin{Bem}
It is easy to see that the composition is associative, and the support of the composition from signed subsets $X$ and $Y$ is equal to the union of supports of $X$ and $Y$. 
\end{Bem}
\text{ }
\vspace{0.5cm}\\
\emph{\textbf{Central Arrangement of Hyperplanes}}\vspace{0.5cm}

An arrangement of hyperplanes $\mathcal{H} = \{H_e:e\in E\}$ is \emph{central} if every hyperplane $H_e$ contains the origin $\mathbf{0}$. With the regularity assumption, this is equivalent to $\bigcap_{e \in E}{H_e} = \{\mathbf{0}\}$.

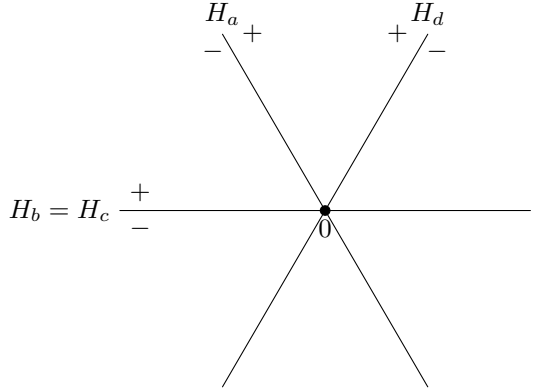
\begin{figure}[h]
\centering
\begin{tikzpicture}[scale=0.9,cap=round,>=latex]
  \coordinate (center) at (0,0);
  \draw (60:3) -- (60:-3);
  \draw (60:3) node[anchor=south]{$H_d$};
  \draw (60:2.7) node[anchor=south east]{$+$};
  \draw (60:2.7) node[anchor=west]{$-$};
  \draw (-60:3) -- (-60:-3);
  \draw (-60:-3) node[anchor=south]{$H_a$};
  \draw (-60:-2.7) node[anchor=east]{$-$};
  \draw (-60:-2.7) node[anchor=south west]{$+$};
  \draw (-3,0) -- (3,0);
  \draw (-3,0) node[anchor=east]{$H_b = H_c$};
  \draw (-2.7,0) node[anchor=south]{$+$};
  \draw (-2.7,0) node[anchor=north]{$-$};

  \filldraw [black] (0,0) circle (2pt);
  \draw (0,0) node[anchor=north]{$0$};
\end{tikzpicture}
\caption{A central arrangement with four hyperplanes.}
\end{figure}

Oriented matroids can be thought of as an abstraction of cell decompositions of such central arrangements:\vspace{0.5cm}
\begin{Def}\label{oriented matroid}
Let $E$ be a finite set. A set $\mathcal{O} \subseteq \{+,-,0\}^E$ is the set of covectors of an \emph{oriented matroid} if and only if $\mathcal{O}$ satisfies the following axioms:
\begin{align*}
(O1)\text{ } &(\emptyset,\emptyset) \in \mathcal{O},\\
(O2)\text{ } &\text{if } X \in \mathcal{O} \text{ then } -X \in \mathcal{O},\\
(O3)\text{ } &\text{if } X,Y \in \mathcal{O} \text{ then } X \circ Y \in \mathcal{O},\\
(O4)\text{ } &\text{if } X,Y \in \mathcal{O} \text{ with } \underline{X} = \underline{Y} \text{ and } e \in S(X,Y) \text{ then there exists } Z \in \mathcal{O} \text{ such that }\\
&Z_e=0 \text{ and } Z_f = {(X \circ Y)}_f = {(Y \circ X)}_f \text{ for all } f \notin S(X,Y).
\end{align*}
The members of $\mathcal{O}$ with maximal support are called the \emph{topes} of $\mathcal{O}$, and the subset $\mathcal{T(O)}$ of $\mathcal{O}$ containing all topes is the \emph{tope set} of $\mathcal{O}$.
\end{Def}
\vspace{0.2cm}

There are also alternative axioms for oriented matroids employing vectors, circuits and so on, see \cite[chapter 3]{bjoerner}. We focus here on the covector axioms. Instead of saying "$\mathcal{O}$ is the collection of covectors of an oriented matroid" we will often simply write "$\mathcal{O}$ is an oriented matroid".
\begin{Kor}\label{tope EqS}
If $\mathcal{O}$ is an oriented matroid on E, then $\mathcal{T(O)}$ has the property that all of its member have the same support.
\end{Kor}
\begin{proof}
By $(O3)$ we have $T \circ T' \in \mathcal{O}$ for all $T,T' \in \mathcal{T(O)} \subseteq \mathcal{O}$. By Remark 1.1, the support of $T \circ T'$ is the union of supports of $T$ and $T'$. Suppose for a contradiction that $\underline{T} \neq \underline{T'}$, then the member $T \circ T'$ of $\mathcal{O}$ witnesses that the supports of $T$ and $T'$ are not maximal, contradicting the definition of a tope.
\end{proof}

From the set of topes $\mathcal{T(O)}$ one can retrieve the entire oriented matroid $\mathcal{O}$. This was first observed by A. Mandel, see \cite[4.2.13]{bjoerner}:

\begin{Thm}[\textbf{Mandel}]
Let $\mathcal{O}$ be an oriented matroid on $E$. Then $\mathcal{O}$ is determined by its tope set $\mathcal{T(O)}$ via
\[\mathcal{O} =\{V \in \{+,-,0\}^{E} : V \circ \mathcal{T} \subseteq \mathcal{T} \}.\]
\end{Thm}

Using this theorem, an axiomatic treatment of affine oriented matroids has been attempted by J.Karlander in his PhD thesis \cite{karlander}. His manuscript was cited either as a preprint (KTH Stockholm) from \cite{bjoerner}, \cite{ck2} or as "to appear in the Eur. J. Combinatorics" for some time in the nineties, e.g. in \cite{ck1} and \cite{ck3}. A document from May 1995 (Combinatorics at ETH) still cites it in this way \cite{KTH}. The paper announced, however, never saw the light of publication, presumably because the final version was either retracted or never delivered.

This is a particularly unfortunate situation since the Karlander thesis contains much valuable material in regard to affine oriented matroids. It is the aim of this paper to revive the interest in the subject and to show that the main theorem of Karlander holds true, despite the fact that there is a fatal flaw in his proof.

\section{Getting started (following Karlander)}
Let $\mathcal{O}$ be an oriented matroid on $E$. We fix some $g \in E$, and let $g^{+}$ be the signed subset $(\{g\},\emptyset)$. Then the set $\mathcal{O}_{g^{+}}$ consists of those covectors in $\mathcal{O}$ whose sign at $g$ is $+$. For a given $X \in \mathcal{O}_{g^{+}}(E - \{g\})$ let $[X,+]$ be the signed subset $(X^{+} \cup \{g\}, X^{-})$, $[X,-]$ and $[X,0]$ are defined in the same way.
\vspace{0.2cm}
\begin{Def}\label{AOM}
Let $E$ be a finite set. A set $\mathcal{W}$ of signed subsets of $E$ is an \emph{affine oriented matroid} (or \emph{affine sign vector system}) if and only if there is an oriented matroid $\mathcal{O}$ on $E$ such that $\mathcal{W} = \mathcal{O}_{g^{+}}(E - \{g\})$.
\end{Def}\text{ }\\
\emph{\textbf{From Affine Arrangements to Crentral Arrangements}}\vspace{0.5cm}

As in the case of oriented matroids, affine oriented matroids can be considered as an abstraction of cell decompositions of affine arrangements. As mentioned in Section 1, central arrangements are only a special type of hyperplane arrangement, but every affine arrangement can be extended into a larger central arrangement as follows:\vspace{0.2cm}

If $\mathcal{H} = \{H_e:e \in E - \{g\}\}$ is an affine arrangement in $\mathbf{R}^{d-1}$, then we can embed $\mathcal{H}$ by the map $x \in \mathbf{R}^{d-1} \mapsto (x,1) \in \mathbf{R}^{d}$. The new coordinate corresponds to $g$ in $E$. For every $e \in E - \{g\}$ let $H'_e$ be the hyperplane in $\mathbf{R}^{d}$ spanned by all points $(x,1)$ where $x$ belongs to $H_e$, and set $H'_g = \{x \in \mathbf{R}^{d}:x_g = 0\}$.

Certainly, the new hyperplane arrangement $\mathcal{H'} = \{H'_e : e \in E\}$ is central. Any assignment of signs to sides of hyperplanes in $\mathcal{H}$ extends itself to the sides of hyperplanes in  $\mathcal{H'}$. The plus side of $H'_g$ is chosen arbitrarily. Let $\mathcal{O}$ denote the corresponding oriented matroid. Then the set of sign vectors with respect to $\mathcal{H}$ in $\mathbf{R}^{d-1}$ is exactly the affine oriented matroid $\mathcal{W} = \mathcal{O}_{g^{+}}(E - \{g\})$.\vspace{0.2cm}

To help us present the characterization of affine oriented matroids, we first introduce some new notations:

\begin{Def}
Let $E$ be a finite set. For $X,Y \in \{+,-,0\}^E$ with $\underline{X}=\underline{Y}$ and $X \neq Y$, we define\vspace{0.2cm}\\
$\bullet$	the \emph{e-elimination set} of $X$ and $Y$ for some $e\in S(X,Y)$ by
\[I_e(X,Y) = \{V \in \{+,-,0\}^E : \underline{V} \subseteq \underline{X} - \{e\} $ and $ V_f=X_f $ for all $ f \notin S(X,Y)\},\]
$\bullet$	the \emph{elimination set} of $X$ and $Y$ by
\[I(X,Y) = \bigcup_{e \in S(X,Y)} I_e(X,Y),\]
$\bullet$	the \emph{equal support set} of $X$ and $Y$ by
\[B(X,Y) = \{V \in \{+,-,0\}^E : V \notin \{X,Y\}, \underline{V} = \underline{X} $ and $ V_f=X_f $ for all $ f \notin S(X,Y)\}.\]
\end{Def}

\begin{Bem}\label{Reformulate}
Let $X,Y \in \{+,-,0\}^E$ be members of an oriented matroid $\mathcal{O}$ with $\underline{X} = \underline{Y}$ and $\text{ } X \neq Y$.\vspace{0.2cm}\\
(1) The sign vector $Z$ qualifies as a member of $I_e(X,Y)$ if it fulfills all requirements of $(O4)$ except that of belonging to $\mathcal{O}$. Thus one can rephrase $(O4)$ as follows:
\[(O4) $ if $ X,Y \in \mathcal{O} \text{ with } \underline{X} = \underline{Y}$ and $X \neq Y$ then $ I_e(X,Y) \cap \mathcal{O} \neq \emptyset \text{ for all } e\in S(X,Y).\]
In the same way we formulate an axiom that is weaker than $(O4)$ in general
\[(O4') $ if $ X,Y \in \mathcal{O} \text{ with } \underline{X} = \underline{Y} $ and $ X \neq Y $ then $ I(X,Y) \cap \mathcal{O} \neq \emptyset.\]
(2) Every member of $B(X,Y)$ has the same support as $X$, whereas the support of members of $I(X,Y)$ is smaller. Therefore
\[B(X,Y) \cap I(X,Y) = \emptyset.\]
The sign vectors $V$ that share the same sign with $X$ and $Y$ at each non-separation coordinate can be classified according to their signs at separation coordinates: Either all those signs are nonzero, i.e. $V \in \{X,Y\} \text{ } \dot{\cup} \text{ } B(X,Y)$, or some zero may occur, whence $V$ belongs to $I(X,Y)$:
\[\{V \in \{+,-,0\}^E : V_f=X_f \text{ for all } f \notin S(X,Y)\} = \{X,Y\} \text{ } \dot{\cup} \text{ } B(X,Y) \text{ } \dot{\cup} \text{ } I(X,Y).\]
\end{Bem}

\begin{Bsp}
Let $E=\{1,2,3,4\}, X=[+,-,+,0]=(\{1,3\},\{2\}), Y=[-,+,+,0]=(\{2,3\},\{1\}).$ Then
\begin{align*}
&S(X,Y)=\{1,2\},\\
&I_1(X,Y)=\{[0,-,+,0],[0,0,+,0],[0,+,+,0]\},\\
&I(X,Y)=\{[0,-,+,0],[0,0,+,0],[0,+,+,0],[-,0,+,0],[+,0,+,0]\},\text{ and}\\
&B(X,Y)=\{[+,+,+,0],[-,-,+,0]\}.
\end{align*}
\end{Bsp}

\begin{Lem}\label{equi}
Assuming $(O3)$, the axioms $(O4)$ and $(O4')$ are equivalent.
\end{Lem}
\begin{proof}
$(O4) \Rightarrow (O4')$ is trivial since $I_e(X,Y) \subseteq I(X,Y)$.

To prove the reverse implication let $\mathcal{O} \subseteq \{+,-,0\}^E \text{ satisfy }(O3) \text{ and } (O4')$. Suppose for a contradiction that $\mathcal{O}$ does not satisfy $(O4)$. Choose $X,Y \in \mathcal{O}$ with $\underline{X} = \underline{Y}$ and $\text{ } X \neq Y$ such that $|S(X,Y)|$ is minimal and $I_e(X,Y) \cap \mathcal{O} = \emptyset$ for some $e \in S(X,Y)$. 

By $(O4')$, there exists some $Z \in I(X,Y) \cap \mathcal{O}$. 
If $Z_e=0$, then $Z \in I_e(X,Y) \cap \mathcal{O}$, which is a contradiction.
Otherwise $e \in \underline{Z}$, and with out loss of generality we can assume $Z_e = X_e \neq 0$. Pick $g \in S(X,Y) - \{e\}$ with $Z_g=0$. By $(O3)$, $X'= Z \circ Y$ is a member of $\mathcal{O}$ satisfying\vspace{0.1cm}\\
(1) $\underline{X}' = \underline{Y}$,\vspace{0.1cm}\\
(2) $X' \neq Y$,\vspace{0.1cm}\\
(3) $S(X',Y) \subset S(X,Y)$, and hence\vspace{0.1cm}\\
(4) $I_e(X',Y) \subseteq I_e(X,Y)$.\\
This contradicts the minimality of $S(X,Y)$.
\end{proof}\vspace{0.2cm}
The axiomatization of affine oriented matroids employ three axioms, of which the first two resemble axioms $(O3)$ and $(O4)$. The last one is given in terms of a particular set $\mathcal{P(W)}$ of sign vectors, to be defined next.

\begin{Def}
Let $E$ be a finite set and $X,Y$ be signed subsets of $E$. Then the \emph{sum} $X + Y$ of $X$ and $Y$ is given by
    \[(X + Y)_e = \begin{cases}
    0, & e \in S(X,Y), 	\\
    (X \circ Y)_{e}, & otherwise.
    \end{cases}\]
\end{Def}
Note that if $\underline{X} = \underline{Y}$ holds, then the sign vector $(X + Y)$ is the $\preceq$-minimal member of $I(X,Y)$.
\begin{Def}
Let $E$ be a finite set and $\mathcal{W} \subseteq \{+,-,0\}^E.$ Then
\begin{align*}
\bullet\text{ }&\sym(\mathcal{W})=\{V \in \{+,-,0\}^E : \pm V \in \mathcal{W}\} \text{ is the maximal symmetric subset of $\mathcal{W}$},\\
\bullet\text{ }&\asym(\mathcal{W})=\{V \in \{+,-,0\}^E: V \in \mathcal{W},-V \notin \mathcal{W}\} \text{ is the complement of $\sym(\mathcal{W})$ in $\mathcal{W}$},\\
\bullet\text{ }&\mathcal{P}(\mathcal{W})=\{X+(-Y):  X,Y \in \asym(\mathcal{W}),\underline{X}=\underline{Y} \text{ and } I(X,-Y) \cap \mathcal{W}=I(-X,Y) \cap \mathcal{W}=\emptyset\}\\
&\text{is the set of \emph{parallel vectors} of $\mathcal{W}$}.
\end{align*}
\end{Def}\text{ }\\
\emph{\textbf{Sign Vectors corrsponding to Points at Infinity}}\vspace{0.5cm}

If $\mathcal{W}$ is generated by an affine arrangement, then the set $\mathcal{P(W)}$ consists of the sign vectors corresponding to virtual points at infinity: Each member of $\mathcal{P(W)}$, except for $(\emptyset,\emptyset)$, represents a parallel class of hyperplanes, resembling the point at infinity from projective geometry. Figure 5 illustrates this idea with a minimal example.\vspace{0.5cm}

\begin{figure}[h]
\centering
\begin{tikzpicture}[scale=0.9,cap=round,>=latex]
  \coordinate (center) at (0,0);
  \draw (0,-3) -- (0,2);
  \draw (0,2) node[anchor=south]{$H_a$};
  \draw (0,1.7) node[anchor=south east]{$+$};
  \draw (0,1.7) node[anchor=south west]{$-$};
  \draw (-3,1) -- (3,1);
  \draw (-3,1) node[anchor=east]{$H_b$};
  \draw (-2.7,1) node[anchor=south east]{$-$};
  \draw (-2.7,1) node[anchor=north east]{$+$};
  \draw (-3,-1) -- (3,-1);
  \draw (-3,-1) node[anchor=east]{$H_c$};
  \draw (-2.7,-1) node[anchor=south east]{$-$};
  \draw (-2.7,-1) node[anchor=north east]{$+$};

  \filldraw [black] (-1.5,-1) circle (2pt);
  \draw (-1.5,-1) node[anchor=north]{$X$};
  \filldraw [black] (1.5,-1) circle (2pt);
  \draw (1.5,-1) node[anchor=north]{$Y$};

  \draw (-3.7,1)[dotted] ..controls (-5.5,0.8).. (-6.1,0);
  \draw (-6.1,0)[dashed] -- (-7.3,1.5);
  \draw (-3.7,-1)[dotted] ..controls (-5.5,-0.8).. (-6.1,0);
  \draw (-6.1,0)[dotted] -- (-7.3,-1.5);

  \filldraw [black] (-6.1,0) circle (2pt);
  \draw (-6.1,0) node[anchor=east]{$P = X + (-Y)$};
  \filldraw [black] (-6.9,1) circle (2pt);
  \draw (-6.9,1) node[anchor=west]{$-Y$};

  \draw (5,-0.5) node[anchor=east]{$X$};
  \draw (5,-0.5) node[anchor=west]{$= [+,+,0],$};
  \draw (5,-1) node[anchor=east]{$Y$};
  \draw (5,-1) node[anchor=west]{$= [-,+,0],$};
  \draw (5,-1.5) node[anchor=east]{$-Y$};
  \draw (5,-1.5) node[anchor=west]{$= [+,-,0],$};
  \draw (5,-2) node[anchor=east]{$P$};
  \draw (5,-2) node[anchor=west]{$ = [+,0,0].$};
\end{tikzpicture}
\vspace{0.5cm}
\caption{Illustration of $P$ as a point at infinity.}
\end{figure}
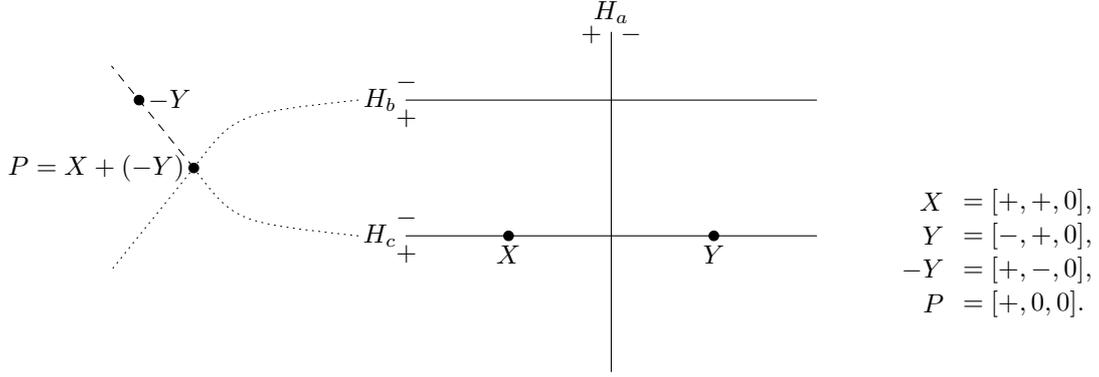

Let $E=\{a,b,c\}$ be the ground set and $\mathcal{H} = \{H_a,H_b,H_c\}$ where $H_b \parallel H_c$ and $H_a$ meets both of them, as shown in Figure 5. And let $\mathcal{W} = \sym(\mathcal{W}) \cup \asym(\mathcal{W})$ be the sign vector system corresponding to $\mathcal{H}$ where\vspace{0.2cm}
\begin{align*}
\sym(\mathcal{W}) = \{&[0,-,-],[+,-,-],[+,+,+]\text{ and their opposites}\},\\
\asym(\mathcal{W}) = \{&[0,0,-],[0,+,0],[0,+,-],[+,0,-],[-,0,-],[+,+,0],\\
&[-,+,0],[+,+,-],[-,+,-]\}.
\end{align*}

We wish to find a pair $X$ and $Y$ signifying cells within $H_c$ such that $P = X + (-Y)$ presents the parallel class of $H_c$.

Since $\mathcal{W}$ corresponds to $\mathcal{H}$, every cell of $\mathcal{H}$ (in the case of Figure 5, these are the intersection points of pairs of hyperplanes, segments and open regions) should have its own sign vector in $\mathcal{W}$.
Thus, if $H_b$ meets $H_c$ at the plus side of $H_a$ as shown with the dotted/dashed lines, then the dashed segment must have its sign vector $-Y$ in $\mathcal{W}$, or equivalently $Y \in \sym(\mathcal{W})$. By applying the same argument to the minus side of $H_a$, this conclusion holds for $-X$, meaning that we have the implication "if $H_b$ meets $H_c$, then at least one of $X$ and $Y$ must be in $\sym(\mathcal{W})$". Notice that the fact "$H_b$ is parallel to $H_c$" can be expressed as "$H_b$ and $H_c$ have no point in common on either side of $H_a$". Therefore, if we wish to construct some $P = X + (-Y)$ that represents this parallel class, then $X$ and $Y$ must be members of $\asym(\mathcal{W})$.

Moreover, the point at infinity should belong exactly to those members of the parallel class represented by that point. In our case, this means $P^0 = \{e \in E : H_e \parallel H_c\}$, i.e. $P$ belongs to the intersection of all hyperplanes in the parallel class of $H_c$. Therefore, any hyperplane having a nonempty intersection with $H_c$ (and therefore with every hyperplane in the parallel class of $H_c$) must have $X$ and $-Y$ on the same side, or equivalently, such a hyperplane must separate $X$ and $Y$. This motivates the other requirement $I(X,-Y) \cap \mathcal{W} = \emptyset$.

Indeed, the intersection of $H_a$ and the segment between $X$ and $Y$ in our example has the sign vector $[0,+,0] \in \mathcal{W}$. It is also easy to see that $[0,+,0]$ belongs to $I(X,Y)$. More generally, if $U$ and $V$ correspond to cells belonging to the same hyperplane $H_f$ and if some hyperplane $H_e$ separates $U$ and $V$, then the sign vector corresponding to $H_e \cap H_f$ must be a member of $I(U,V)$. In other words, to construct a member $P$ of $\mathcal{P(W)}$, we must ensure that the segment between $X$ and $-Y$ only meets hyperplanes in the same parallel class as $H_b$ but no others. Or put differently, $X$ and $Y$ are chosen so that the segment between $X$ and $Y$ meets every hyperplane not belonging to the parallel class $H_c$.

By applying the above argument to the minus side of $H_a$, one arrives at the analog conclusion for $X$ and $I(-X,Y)$. Every member of $\mathcal{P(W)}$ thus represents, together with its opposite, one parallel class of hyperplanes. Now we can state Karlander's axiomatization of affine oriented matroids.
\begin{Thm}\label{karlander} 
A set $\mathcal{W} \subseteq\{+,-,0\}^E$ is an affine oriented matroid if and only if $\mathcal{W}$ satisfies
\begin{align*}
(A1) &\text{ if } X,Y \in \mathcal{W} \text{ then } X \circ (\pm Y) \in \mathcal{W},\\
(A2) &\text{ if } X,Y \in \mathcal{W} \text{ with }\underline{X}=\underline{Y} \text{ then } I_e(X,Y) \cap \mathcal{W} \neq \emptyset \text{ for all $e \in S(X,Y)$},\\
(A3) &\text{ } \mathcal{P}(\mathcal{W}) \circ \mathcal{W} \subseteq \mathcal{W}.
\end{align*}
\end{Thm}

To motivate the axiom $(A3)$, recall that $\mathcal{P(W)}$ serves as the set of points at infinity. If we choose  $X \in \mathcal{W}$ and $P \in \mathcal{P}$, then the segment between $X$ and $P$ must pass through a cell corresponding to $P \circ X$ before going to infinity. Therefore, $P \circ X$ must belong to $\mathcal{W}$. The proof of Theorem 2.1 is somewhat involved and will be provided at the end of this section. We still need some preparation. 

\begin{Def}
Let $E$ be a finite set and $\mathcal{W} \subseteq \{+,-,0\}^E.$ Then
\[\mathcal{N(W)} =\{N \in \{+,-,0\}^{E}: (\pm N) \circ \mathcal{W} \subseteq \mathcal{W}\}.\]
\end{Def}
While $\mathcal{P}(\mathcal{W})$ is needed for the axiomatization, $\mathcal{N}(\mathcal{W})$ is used to reconstruct an oriented matroid $\mathcal{O}$ from a given affine oriented matroid $\mathcal{W}$. The theorem of Mandel suggests how this might be done: we can reconstruct the tope set of $\mathcal{O}$ from $\mathcal{W}$ as shown in the next lemma.
\vspace{0.2cm}
\begin{Lem}\label{N property}
Let $\text{ } \mathcal{W} \subseteq \{+,-,0\}^E$ be an affine oriented matroid obtained from an oriented matroid $\mathcal{O} \subseteq \{+,-,0\}^{E \cup \{g\}}$ as in Definition \ref{AOM}. Then
\[\mathcal{N}(\mathcal{W})=\{N \in \{+,-,0\}^E : [N,0] \in \mathcal{O}\}.\]
\end{Lem}
\begin{proof}
Let $\mathcal{T(W)}$ comprise the members of $\mathcal{W}$ with maximal support in $\mathcal{W}$. Note that for all $T \in \{+,-,0\}^{E}$ we have $T \in \mathcal{T(W)}$ if and only if $[T,+] \in \mathcal{T(O)}$.

Let $\mathcal{N'(W)} = \{N \in \{+,-,0\}^E : (\pm N) \circ \mathcal{T(W)} \subseteq \mathcal{T(W)}\}$. First we claim that $\mathcal{N'(W)} = \{N \in \{+,-,0\}^E : [N,0] \in \mathcal{O}\}$.
\vspace{0.2cm}

For the forward inclusion let $N \in \mathcal{N'(W)}$. By the Theorem of Mandel, it is sufficient to show $[N,0] \circ \mathcal{T(O)} \subseteq \mathcal{T(O)}$. Let $T'$ be a tope of $\mathcal{O}$. If $T' = [T,+]$, then $T \in \mathcal{T(W)}$ and we have $N \circ T \in \mathcal{T(W)}$, thus $[N,0] \circ T' = [N \circ T,+] \in \mathcal{T(O)}$. Otherwise, if  $T' = [T,-]$, we must have $-T \in \mathcal{T(W)}$ by $(O2)$. Using the same equality we conclude $[-N,0] \circ [-T,+] = [-(N \circ T),+] \in \mathcal{T(O)}$, and, by $(O2)$ once again, $[N,0] \circ T' = -[-(N \circ T),+] \in \mathcal{T(O)}$. Note that $g$ must be contained in the support of the tope set of $\mathcal{O}$, so that we have  $[N,0] \circ \mathcal{T(O)} \subseteq \mathcal{T(O)}$.

For the backward inclusion let N be a signed subset of E such that $[N,0] \in \mathcal{O}$. By $(O2)$, $[-N,0] \in \mathcal{O}$. For arbitrary $T \in \mathcal{T(W)}$ we have $[T,+] \in \mathcal{O}$. $(O3)$ implies $[(\pm N) \circ T,+] = [\pm N,0] \circ [T,+] \in \mathcal{T(O)}$. Therefore $(\pm N) \circ T \in \mathcal{T(W)}$ and thus $N \in \mathcal{N'(W)}$.
\vspace{0.2cm}

With the claim above we only have to show that $\mathcal{N}(\mathcal{W}) = \mathcal{N}'(\mathcal{W})$. The forward inclusion is trivial. For the backward inclusion let $N \in \mathcal{N'(W)}$, and $V \in \mathcal{W}$ be arbitrary. Using the Theorem of Mandel we have $(\pm N) \circ V \in \mathcal{W}$ if and only if $[(\pm N) \circ V,+] \circ \mathcal{T(O)} = [\pm N,0] \circ ([V,+] \circ \mathcal{T(O)}) \subseteq \mathcal{T(O)}$. By definition of $\mathcal{W}$ and the claim above, $[\pm N,0]$ and $[V,+]$ are members of $\mathcal{O}$. Therefore $[\pm N,0] \circ ([V,+] \circ \mathcal{T(O)}) \subseteq [\pm N,0] \circ \mathcal{T(O)} \subseteq \mathcal{T(O)}$.
\end{proof}

\begin{Bem}
Let $\mathcal{W} \subseteq \{+,-,0\}^E$ be an affine sign vector system obtained from an oriented matroid $\mathcal{O}$ on $E \cup \{g\}$ as in Definition \ref{AOM}. Then the set of members of $\mathcal{O}$ with $g$ contained in their supports is given either by $\mathcal{W}$ or by $-\mathcal{W}$. Using Lemma \ref{N property}, one is able to find the remaining members of $\mathcal{O}$, those having sign $0$ in $g$. Let
\[\mathcal{W}^\dagger = \{[W,+] : W \in \mathcal{W}\} \cup \{[-W,-] : W \in \mathcal{W}\} \cup \{[N,0] : N \in \mathcal{N}(\mathcal{W})\}.\]
Then $\mathcal{W}$ is an affine oriented matroid if and only if $\mathcal{W}^\dagger$ is an oriented matroid. Obviously, $\mathcal{W}^\dagger$ is the unique oriented matroid with $\mathcal{W}$ as its affine oriented matroid.
\end{Bem}
\vspace{0.2cm}
The next goal is to clarify the relation between $\mathcal{P}(\mathcal{W})$ and $\mathcal{N}(\mathcal{W})$, which is important for the proof of Theorem 2.1. Before that, a property of $B(X,Y)$ has to be recorded.
\vspace{0.2cm}
\begin{Lem}\label{B property}
If $\mathcal{W}$ is an affine oriented matroid and $X,Y \in \mathcal{W}$ with $I(X,-Y) \cap \mathcal{W} = \emptyset$, then $B(X,-Y) \cap \mathcal{W} = \emptyset$.
\end{Lem}
\begin{proof}
Suppose for a contradiction that $B(X,-Y) \cap \mathcal{W} \neq \emptyset$. For $Z \in B(X,-Y) \cap \mathcal{W}$ we have $S(X,Z) \neq \emptyset$ as $\underline{Z} = \underline{X}$ and $Z \neq X$. Moreover, we have $I(X,Z) \subseteq I(X,-Y)$ since $Z$ shares the same sign with $X$ and $-Y$ at each non-separation coordinate. On the other hand, it is easy to see that $\mathcal{W}$ inherits $(O4)$ from its oriented matroid. Thus we conclude that $ I(X,Z) \cap \mathcal{W} \neq \emptyset$, contradicting $I(X,-Y) \cap \mathcal{W} = \emptyset$.
\end{proof}

\begin{Prop}\label{N zerlegung}
If $\mathcal{W}$ is an affine oriented matroid, then $\mathcal{N}(\mathcal{W}) = \sym(\mathcal{W}) \text{ } \dot{\cup} \text{ } \mathcal{P}(\mathcal{W})$.
\end{Prop}
\begin{proof}
Let $\mathcal{O}$ be the oriented matroid of $\mathcal{W}$.
First we claim that it suffices to verify the following four properties\vspace{0.1cm}\\
$(1)\text{ }\mathcal{W} \cap \mathcal{N}(\mathcal{W}) = \sym(\mathcal{W})$,\vspace{0.1cm}\\
$(2)\text{ }\mathcal{P}(\mathcal{W}) \subseteq \mathcal{N}(\mathcal{W})$,\vspace{0.1cm}\\
$(3)\text{ }\mathcal{N}(\mathcal{W}) \subseteq \sym(\mathcal{W}) \text{ } \cup \text{ } \mathcal{P}(\mathcal{W})$, and\vspace{0.1cm}\\
$(4)\text{ }\mathcal{P}(\mathcal{W}) \cap \sym(\mathcal{W}) = \emptyset$.\vspace{0.1cm}\\
Indeed, $(1)$, $(2)$ and $(4)$ imply that $\sym(\mathcal{W}) \text{ }\dot{\cup}\text{ } \mathcal{P}(\mathcal{W}) \subseteq \mathcal{N}(\mathcal{W})$, and together with $(3)$ this yields the desired equality.
\vspace{0.5cm}

$(1)$
Let $V \in \mathcal{W} \cap \mathcal{N}(\mathcal{W})$. Then 
\begin{align*}
[V,+],[V,0] \in \mathcal{O}
&\stackrel{(O2)}{\Rightarrow} [-V,-] \in \mathcal{O}\\
&\stackrel{(O3)}{\Rightarrow} [V,0] \circ [-V,-] = [V,-] \in \mathcal{O}\\
&\stackrel{(O2)}{\Rightarrow} [-V,+] \in \mathcal{O}\\
&\Rightarrow -V \in \mathcal{W}\\
&\Rightarrow V \in \sym(\mathcal{W}).
\end{align*}

Conversely, let $V \in \sym(\mathcal{W}) \subseteq \mathcal{W}$. Then we have
\begin{align*}
[\pm V,+] \in \mathcal{O}
&\stackrel{(O3)}{\Rightarrow} [V,-] \in \mathcal{O}\\
&\stackrel{(O4)}{\Rightarrow} [V,0] \in \mathcal{O}\\
&\Rightarrow V \in \mathcal{W} \cap \mathcal{N}(\mathcal{W}).
\end{align*}

$(2)$
Let $P=X+(-Y) \in \mathcal{P}(\mathcal{W})$. Then $[X,+],[Y,+],[-Y,-] \in \mathcal{O}$. Thus $I_g([X,+],[-Y,-]) \cap \mathcal{O} \neq \emptyset$.
Pick some $[Z,0] \in I_g([X,+],[-Y,-]) \cap \mathcal{O}$. We claim that $P = Z$.
Since $Z$ shares the same sign with $X$ and $-Y$ at each non-separation coordinate, we have $Z \in B(X,-Y) \text{ }\dot{\cup}\text{ } I(X,-Y) \text{ }\dot{\cup}\text{ } \{X,-Y\}$. One can argue that $Z \in I(X,-Y)$ as follows.

Notice that Lemma \ref{N property} implies $Z \in \mathcal{N(W)}$.
Hence, if $Z$ or $-Z$ is a member of $\mathcal{W}$ then it is contained in $\sym(\mathcal{W})$, and hence $Z \notin \{X,-Y\}$.
And $Z$ is not contained in $B(X,-Y)$ since otherwise we would have $Z = Z \circ X \in \mathcal{W}$,
meaning that $B(X,-Y) \cap \mathcal{W} \neq \emptyset$, which would contradicting Lemma \ref{B property}. Therefore, $Z$ must be a member of $I(X,-Y)$.

In other words $P = X + (-Y) \preceq Z$. It suffices to show $Z \circ X = X$ and $Z \circ (-Y) = -Y$.
By Lemma 2.2, $Z \circ X$ and $(-Z) \circ Y$ are members of $\mathcal{W}$. Since $B(X,-Y) \cap \mathcal{W} = B(-X,Y) \cap \mathcal{W} = \emptyset$,
one has $Z \circ X \notin B(X,-Y)$ and $(-Z) \circ Y \notin B(-X,Y)$. The latter is equivalent to $Z \circ (-Y) \notin B(X,-Y)$.
Note that both of  $Z \circ X$ and $Z \circ (-Y)$ have the same support as $X$ and share the same sign with $X$ and $-Y$ at each non-separation coordinate as $Z$  does. Therefore, $Z \circ X$ and $Z \circ (-Y)$ must be contained in $\{X,-Y\}$. Since $Z \in I(X,-Y)$, we have $\underline{Z} \subset \underline{X}$ and thus $Z \circ X \neq -Y$,
implying $Z \circ X = X$ and similarly $Z \circ (-Y) = -Y$.

$(3)$
Let $N \in \mathcal{N}(\mathcal{W}) - \sym(\mathcal{W})$.
It suffices to show $N \in \mathcal{P}(\mathcal{W})$.
Choose $V \in \mathcal{W}$ such that $|\underline{V} - \underline{N}|$ is minimal.
Note that for all $V \in \mathcal{W}$ we have $\underline{V} \not\subseteq \underline{N}$. Otherwise, by definition of $\mathcal{N(W)}$ we have $N = N \circ V \in \mathcal{W} \cap \mathcal{N}(\mathcal{W}) = \sym(\mathcal{W})$, contradicting $N \notin \sym(\mathcal{W})$. 

Let $X = N \circ V$ and $Y = (-N) \circ V$. Then $X,Y \in \mathcal{W}$ by definition of $\mathcal{N(W)}$, and we get $N = X + (-Y)$. In order to show $X + (-Y) \in \mathcal{P}(\mathcal{W})$, it suffices to show $X \in \asym(\mathcal{W})$ and $I(X,-Y) \cap \mathcal{W} = \emptyset$ since the argument for $Y \in \asym(\mathcal{W})$ and $I(-X,Y) \cap \mathcal{W} = \emptyset$ is similar.

Suppose for a contradiction that $I(X,-Y) \cap \mathcal{W} \neq \emptyset$, say $Z \in I(X,-Y) \cap \mathcal{W}$ with $Z_e = 0$ for some $e \in S(X,-Y)$.
We claim that $|\underline{Z} - \underline{N}|$ is less than $|\underline{V} - \underline{N}|$.
Indeed, $\underline{Z}$ is a proper subset of $\underline{X}$ as $Z \in I(X,-Y)$. This implies $|\underline{Z} - \underline{N}| \le |\underline{X} - \underline{N}| = |\underline{V} - \underline{N}|$.
Furthermore, $e$ is a member of $S(X,-Y)$ and thus not contained in $S(X,Y) = \underline{N}$, therefore $|\underline{Z} - \underline{N}| < |\underline{V} - \underline{N}|$, contradicting minimality.

If $X \in \sym(\mathcal{W})$, then $-X$ must belong to $\mathcal{W}$. By definition of $\mathcal{N(W)}$, we have $N \circ (-X) = -Y \in \mathcal{W}$, which implies $I(X,-Y) \cap \mathcal{W} \neq \emptyset$ by $(O4)$ and leads to a contradiction as in the argument above.

$(4)$
It is easy to see that every $P = X + (-Y) \in \mathcal{P}(\mathcal{W})$ is a member of the elimination set $I(X,-Y)$. Therefore, we conclude that $\mathcal{W} \cap \mathcal{P}(\mathcal{W}) = \emptyset$ and (4) holds as an immediate consequence.  
\end{proof}

\begin{Bem}\label{remark suff}
If $\mathcal{W}$ is an affine oriented matroid obtained from some oriented matroid $\mathcal{O}$, then $(A1)$ and $(A2)$ are inherited from $\mathcal{O}$. By Proposition 2.1 we have $\mathcal{P}(\mathcal{W}) \subseteq \mathcal{N}(\mathcal{W})$, and $(A3)$ follows from Lemma \ref{N property}.
\end{Bem}
\vspace{0.2cm}
This remark ensures sufficiency of the axiomatization. In order to see necessity, yet another technical lemma has to be established.
\vspace{0.2cm}
\begin{Lem}\label{Zauber}
Let $\mathcal{W} \subseteq \{+,-,0\}^E$ satisfy $(A1)$,$(A2)$ and $(A3)$. Further let $P = U + (-U') \in \mathcal{P}(\mathcal{W})$ with $U,U' \in \asym(\mathcal{W})$. Then $U$ conforms to every $Z \in \mathcal{W}$ with $\underline{Z} \subseteq \underline{U}$\vspace{0.2cm} at  $\underline{U} - \underline{P}$, that is,
\[Z_f=U_f \text{ for all } f \in \underline{U} - \underline{P}.\]
\end{Lem}
\begin{proof}
For such a member of $\mathcal{W}$ there are five possibilities\vspace{0.2cm}\\
(a) $\underline{Z} \subseteq \underline{P}$,\\
(b) $\underline{P} \subseteq \underline{Z}$ and for all $f \in \underline{U} - \underline{P}$ we have $Z_f=U_f$,\\
(c) $\underline{P} \subseteq \underline{Z}$ and for all $f \in \underline{U} - \underline{P}$ we have $Z_f=-U_f$,\\
(d) $\underline{P} \subseteq \underline{Z}$ and there exist $f,h \in \underline{U} - \underline{P}$ such that $Z_f=U_f$ and $Z_h=-U_h$,\\
(e) $\underline{Z} \not\subseteq \underline{P}$ and $\underline{P} \not\subseteq \underline{Z}$.\vspace{0.2cm}\\
We need to show that all cases other than (b) lead to case (b) eventually, or to a contradiction. Note that the case (e) was overlooked (or deemed to be trivial) in the original proof by Karlander.

For (a) notice that $P = P \circ Z \in \mathcal{W}$ by $(A3)$. We thus have $\mathcal{P}(\mathcal{W}) \cap \mathcal{W} \neq \emptyset$, contradicting the proof of Proposition \ref{N zerlegung}{.(4)}.

As to (c), note that $\underline{U} - \underline{P}$ comprises exactly those coordinates at which $U$ differs from $-U'$ because $P = U + (-U')$. Therefore $P \circ Z= -U'$, and by $(A3)$ we conclude  $-U' = P \circ Z \in \mathcal{W}$, which however contradicts that $U' \in \asym(\mathcal{W})$.

For (d) observe that for every $f \notin \underline{U} - \underline{P} = S(U,-U')$ we have $(P \circ Z)_f = P_f = U_f = -U'_f$, meaning that $P \circ Z \in \{U,-U'\} \cup I(U,-U') \cup B(U,-U')$ by Remark 2.1.

On the one hand, we have $\underline{P \circ Z} = \underline{U}$ as in the case (c), whence $P \circ Z \notin I(U,-U')$.
On the other hand, by assumption neither $U$ nor $-U'$ conforms to $P \circ Z$.
Hence $P \circ Z \notin \{U,U'\}$, so that $P \circ Z$ is a member of $B(U,-U')$ and hence we have $I(U,P \circ Z) \cap \mathcal{W} \neq \emptyset$ by $(A2)$.
Since  $P \circ Z$ shares the same sign with $U$ and $-U'$ at each non-separation coordinate, we have $I(U,P \circ Z) \subseteq I(U,-U')$.
This implies $I(U,-U') \cap \mathcal{W} \neq \emptyset$, contradicting the definition of $\mathcal{P(W)}$.

For (e) observe that $\underline{P} \subseteq \underline{P \circ Z} \subseteq \underline{U}$ and $P \circ Z$ has the same sign as $Z$ at $\underline{U} - \underline{P}$. Applying the above proof to $P \circ Z$,
we obtain the property ${(P \circ Z)}_f=U_f$ for all $f \in \underline{U} - \underline{P}$. Since $P_f =0$ for all $f \notin \underline{P}$, we get ${(P \circ Z)}_f=Z_f$ for such $f$ and in particular for those $f \in \underline{U} - \underline{P}$, meaning that $Z$ has also the desired property.
\end{proof}
\vspace{0.2cm}
Note that $Z$ in the above context must be contained in $\asym(\mathcal{W})$.
Suppose for a contradiction that $-Z$ is also a member of $\mathcal{W}$.
Then we could apply Lemma \ref{Zauber} to $-Z$ and conclude that $U_f = Z_f = -Z_f$ for all $f \in \underline{U} - \underline{P}$. This is an evident contradiction as $U_f \neq 0$.

\section{The Proof of Karlander's Main Theorem repaired}
Having established necessity in the main theorem by Remark \ref{remark suff}, it remains to prove suffciency.

\begin{Lem}
Let $\mathcal{W} \subseteq \{+,-,0\}^E$ satisfy $(A1)$, $(A2)$, and $(A3)$. Then $\mathcal{W}$ is an affine oriented matroid.
\end{Lem}

\begin{proof}
Recall the definition of  $\mathcal{W}^\dagger = \{[W,+] : W \in \mathcal{W}\} \cup \{[-W,-] : W \in \mathcal{W}\} \cup \{[N,0] : N \in \mathcal{N}(\mathcal{W})\}$. By Remark 2.2 and Lemma 2.1, it suffices to show that $\mathcal{W}^\dagger$ satisfies $(O1)$, $(O2)$, $(O3)$, and $(O4')$.
\vspace{0.2cm}

As to $(O1)$, choose $V \in \mathcal{W}$ with minimal support.
If $V = (\emptyset,\emptyset) = -V$, then $V \in \sym(\mathcal{W})$ and thus $[V,0] \in  \mathcal{W}^\dagger$.
Suppose that $\underline{V} \neq \emptyset$. 
Observe that the set $I(V,-V) \cap \mathcal{W}$ must be empty, for otherwise, every member of $I(V,-V) \cap \mathcal{W}$ has a smaller support than $V$, contradicting minimality.
$V$ must actually be a member of $\asym(\mathcal{W})$, since otherwise we could apply $(A2)$ to the pair $V,-V$ to find a member of $I(V,-V)$. Thus we conclude that $(\emptyset,\emptyset) = V + (-V) \in \mathcal{P(W)}$ and consequently $\mathcal{W}^\dagger$ satisfies $(O1)$.

It is easy to see that $\mathcal{P}(\mathcal{W})$ is symmetric. Hence $\mathcal{W}^\dagger$ satisfies $(O2)$.

As for $(O3)$ let $V_1,V_2 \in \mathcal{W}^\dagger$. If $V_1,V_2 \notin \{[N,0] : N \in \mathcal{N}(\mathcal{W})\}$, the axiom $(O3)$ is a direct consequence of $(A1)$. We can therefore assume that $V_1 = [N,0]$ for some $N \in \sym(\mathcal{W}) \cup \mathcal{P}(\mathcal{W})$.\vspace{0.2cm}\\
Case 1 of $(O3)$: $V_2 \in \{[W,+] : W \in \mathcal{W}\} \cup \{[-W,-] : W \in \mathcal{W}\}$.

Let $V \in \mathcal{W}$ correspond to $V_2$.
We claim that $V \circ (\pm N) \in \mathcal{W}$ and $(\pm N) \circ V \in \mathcal{W}$. By $(O2)$, this claim is sufficient for $V_1 \circ V_2 \in \mathcal{W}^\dagger$ and $V_2 \circ V_1 \in \mathcal{W}^\dagger$.

If $N \in \sym(\mathcal{W})$, the claim is a consequence of $(A1)$. Otherwise $N \in \mathcal{P}(\mathcal{W})$, by $(A3)$ and the symmetry of $\mathcal{P}(\mathcal{W})$ we have $(\pm N) \circ V \in \mathcal{W}$. Together with the equality \vspace{0.2cm}

$V \circ (\pm N) = (V \circ (\pm N)) \circ V = V \circ ((\pm N) \circ V)$,\vspace{0.2cm}\\
one concludes by $(A1)$ that $V \circ (\pm N) \in \mathcal{W}$.\vspace{0.2cm}\\
Case 2 of $(O3)$: $V_2 = [N',0]$ with $N' \in \sym(\mathcal{W}) \cup \mathcal{P}(\mathcal{W})$.

Then we have $N' \circ \mathcal{W} \subseteq \mathcal{W}$ either by $(A1)$ for those $N' \in \sym(\mathcal{W})$, or by $(A3)$ for those $N' \in \mathcal{P(W)}$. It is enough to show $T := N \circ N' \in \sym(\mathcal{W}) \cup \mathcal{P}(\mathcal{W})$. Assuming $T \notin\sym(\mathcal{W})$, we claim $T \in \mathcal{P}(\mathcal{W})$. Note the following inclusion \vspace{0.2cm}

$T \circ \mathcal{W} = (N \circ N') \circ \mathcal{W} = N \circ (N' \circ \mathcal{W}) \subseteq N \circ \mathcal{W} \subseteq \mathcal{W}$. $(*)$\vspace{0.2cm}\\
Using symmetry of $\sym(\mathcal{W}) \cup \mathcal{P}(\mathcal{W})$, one infers that $-T$ has the same property. 

Choose $V \in \mathcal{W}$ with the minimal $|\underline{V} - \underline{T}|$. The minimal cardinality must be larger than 0, otherwise there exists some $V \in \mathcal{W}$ such that $\pm T = (\pm T) \circ V \in \mathcal{W}$, contradicting the assumption $T \notin \sym(\mathcal{W})$.

Let $X = T \circ V$ and $Y = (-T) \circ V$, so that $T = X + (-Y)$. By $(*)$ we have $X,Y \in \mathcal{W}$ with $X \neq Y$ and $\underline{X} = \underline{Y}$. 
To establish our claim $T \in \mathcal{P(W)}$,
it remains to show $X,Y \in \asym(\mathcal{W})$ and $(I(X,-Y) \cup I(-X,Y)) \cap \mathcal{W} = \emptyset$.

Certainly, $-Y = T \circ (-X)$ and $-X = (-T) \circ (-Y)$. If $X \in \sym(\mathcal{W})$, we can apply $(*)$ to $T$ and $-X$ to conclude that $-Y \in \mathcal{W}$.
Since $|\underline{V} - \underline{T}| \neq 0$, we have $S(X,-Y) \neq \emptyset$, implying that $I(X,-Y) \cap \mathcal{W} \neq \emptyset$ by $(A2)$. Similarly, if $Y \in \sym(\mathcal{W})$, one concludes $I(-X,Y) \cap \mathcal{W} \neq \emptyset$.
Therefore, it suffices to show $(I(X,-Y) \cup I(-X,Y)) \cap \mathcal{W} = \emptyset$.

Suppose for a contradiction that $(I(X,-Y) \cup I(-X,Y)) \cap \mathcal{W}$ is not empty. Then the support of any $Z \in (I(X,-Y) \cup I(-X,Y)) \cap \mathcal{W}$ is a proper subset of support of $X$. By definition of $X,Y$ we have $S(X,-Y) = S(-X,Y) = \underline{V} - \underline{T}$, meaning that $Z \in \mathcal{W}$ has the property $|\underline{Z} - \underline{T}|<|\underline{V} - \underline{T}|$, contradicting minimality.

Finally, to establish $(O4')$, let $V_1,V_2 \in \mathcal{W}^\dagger$ with $V_1 \neq V_2$ and $\underline{V_1} = \underline{V_2}$. It is enough to show $I(V_1,V_2) \cap \mathcal{W}^\dagger \neq \emptyset$.\vspace{0.2cm}\\
Case 1 of $(O4')$: $V_1,V_2 \in \{[W,+] : W \in \mathcal{W}\}$ or $V_1,V_2 \in \{[-W,-] : W \in \mathcal{W}\}$.

In this case, $(O4')$ is a direct consequence of $(A2)$.\vspace{0.2cm}\\
Case 2 of $(O4')$: $V_1 = [X,+] \in \{[W,+] : W \in \mathcal{W}\}$ and $V_2 = [-Y,-] \in \{[-W,-] : W \in \mathcal{W}\}$.

If one of $X,Y$ is a member of $\sym(\mathcal{W})$, say $X$, then $[X,0] \in \mathcal{W}^\dagger$.
Observe that $[X,0]$ also belongs to $I([X,+],[-Y,-])$ where the last coordinate is forced to be $0$ by $(O4')$, whence $I([X,+],[-Y,-]) \cap \mathcal{W^\dagger} \neq \emptyset$.
The case $Y \in \sym(\mathcal{W})$ is similar and thus we conclude that $I(V_1,V_2) \cap \mathcal{W}^\dagger$ is not empty whenever at least one of $X$ and $Y$ belongs to $\sym(\mathcal{W})$. 

If $X,Y \in \asym(\mathcal{W})$ and there is some $Z \in I(X,-Y) \cap \mathcal{W}$, then $[Z,+] \in I([X,+],[-Y,-]) \cap \mathcal{W}^\dagger$ and we are done. The case for $I(-X,Y) \cap \mathcal{W} \neq \emptyset$ is similar.
Thus we can assume that $(I(X,-Y) \cup I(-X,Y)) \cap \mathcal{W} = \emptyset$.
Then $X + (-Y) \in \mathcal{P}(\mathcal{W})$ and consequently $[X + (-Y),0] \in I([X,+],[-Y,-]) \cap \mathcal{W}^\dagger$.

By the assumption $\underline{V_1} = \underline{V_2}$, it is not allowed to have exactly one of $V_1,V_2$ in $\{[N,0] : N \in \sym(\mathcal{W}) \cup \mathcal{P}(\mathcal{W})\}$. Thus the only remaining case is\vspace{0.2cm}\\
Case 3 of $(O4')$: $V_1,V_2 \in \{[N,0] : N \in \sym(\mathcal{W}) \cup \mathcal{P}(\mathcal{W})\}$, say $V_1 = [N_1,0]$ and $V_2 = [N_2,0]$.

Nor is it possible to have exactly one of $N_1, N_2$ in $\sym(\mathcal{W})$.
Otherwise, say $N_1 \in \sym(\mathcal{W})$ and $N_2 \in \mathcal{P}(\mathcal{W})$, we could apply $(A3)$ to get $N_2 = N_2 \circ N_1 \in \mathcal{W}$. This however contradicts $\mathcal{P}(\mathcal{W}) \cap \mathcal{W} = \emptyset$.

For any $Z_1,Z_2 \in \{+,-,0\}^E$ with $\underline{Z_1} = \underline{Z_2}$ we define\vspace{0.2cm}

$\mathcal{A}(Z_1,Z_2) = \{ (X,-Y) \in \{+,-,0\}^E: X,Y \in \asym(\mathcal{W}), \underline{X} = \underline{Y} \text{ and } X,-Y \in I(Z_1,Z_2)\}$\vspace{0.2cm}\\
Depending on the choice of $Z_1$ and $Z_2$, the set $\mathcal{A}(Z_1,Z_2)$ may be empty; but if it is not, then its members can be used to construct members of $\mathcal{P}(\mathcal{W})$.\vspace{0.2cm}\\
Case 3.1 of $(O4')$: $N_1,N_2 \in \sym(\mathcal{W})$.

If there exists some $U \in I(N_1,N_2) \cap \sym(\mathcal{W})$,
then $[U,0] \in \mathcal{W}^\dagger \cap I([N_1,0],[N_2,0])$ and we are done.
So let us assume $I(N_1,N_2) \cap \sym(\mathcal{W}) = \emptyset$ and show that $I(N_1,N_2) \cap \mathcal{P}(\mathcal{W}) \neq \emptyset$.
The plan is to pick some pair $(X,-Y)$ from $\mathcal{A}(N_1,N_2)$ with minimal separation set and prove that $X + (-Y)$ belongs to $I(N_1,N_2) \cap \mathcal{P}(\mathcal{W})$.
In order to witness $\mathcal{A}(N_1,N_2) \neq \emptyset$,
consider $e \in S(N_1,N_2)$.
By (A2) there exist\vspace{0.2cm}

$A_1 \in I_e(N_1,N_2) \cap \mathcal{W}$ \text{ } and \text{ } $A_2 \in I_e(-N_1,-N_2) \cap \mathcal{W}$.\vspace{0.2cm}\\
Let $B_1 = A_1 \circ A_2$ and $B_2 = A_2 \circ A_1$ and thus $\underline{B_1} = \underline{B_2}$ with $e \notin \underline{B_1}$.
Clearly, $B_1$ and $-B_2$ share the same sign with $N_1$ and $N_2$ at each non-separation coordinate as $A_1,-A_2$ do. 
Hence in view of $(A1)$ we conclude that $B_1,-B_2 \in I(N_1,N_2) \cap \mathcal{W}$. The assumption $I(N_1,N_2) \cap \sym(\mathcal{W}) = \emptyset$ thus implies that $B_1$ and $B_2$ belong to $\asym(\mathcal{W})$ and therefore $(B_1,-B_2) \in \mathcal{A}(N_1,N_2)$.

Now choose $(X,-Y) \in \mathcal{A}(N_1,N_2)$ with minimal separation set.
Note for each pair $(X,-Y) \in \mathcal{A}(N_1,N_2)$ we always have $|S(X,-Y)| > 0$,
for otherwise, the assumption $\underline{X} = \underline{Y}$ would give $X = -Y$, contradicting $Y \in \asym(\mathcal{W})$.
Certainly, $X$ and $-Y$ share the same sign with $N_1$ and $N_2$ at each non-separation coordinate of $N_1,N_2$ and so does $P = X + (-Y)$.
It also easy to see that the support of $P$ is a subset of the support of $X$, and hence a proper subset of the support of $N_1$ because $X \in I(N_1,N_2)$.
This establishes $P \in I(N_1,N_2)$.

It remains to show that $P \in \mathcal{P(W)}$.
Recall\vspace{0.2cm}

$\mathcal{P}(\mathcal{W})=\{X+(-Y):  X,Y \in \asym(\mathcal{W}),\underline{X}=\underline{Y} \text{ and } I(X,-Y) \cap \mathcal{W}=I(-X,Y) \cap \mathcal{W}=\emptyset\}.$\vspace{0.2cm}\\
Since $(X,-Y) \in \mathcal{A}(N_1,N_2)$, we only need to check the condition on the elimination sets.
Suppose for a contradiction that there exists $Z \in (I(X,-Y) \cup I(-X,Y)) \cap \mathcal{W}$ with $Z_e = 0$ for some $e \in S(X,-Y) = S(-X,Y)$.

In the case $Z \in I(X,-Y)$ we claim that the pair $Z \circ (-Y)$ and $-Y$ belong to $\mathcal{A}(N_1,N_2)$ and have a separation set smaller than the separation set of $X$ and $-Y$.

Notice that $Z \circ (-Y) \in \mathcal{W}$ is a direct consequence of $(A1)$.
Similarly as above, $Z \circ (-Y)$ share the same sign with $X$ and $-Y$ at non-separation coordinates of $X,-Y$ and in particular at non-separation coordinates of $N_1,N_2$, because the inequality $S(X,-Y) \subseteq S(N_1,N_2)$ holds.
By definition of elimination set we have that $\underline{Z} \subset \underline{X}$ as well as $\underline{X} \subset \underline{N_1}$,
implying that the support of $Z \circ (-Y)$ is a proper subset of the support of $N_1$.
We conclude $Z \circ (-Y) \in I(N_1,N_2) \cap \mathcal{W}$. The assumption $I(N_1,N_2) \cap \sym(\mathcal{W}) = \emptyset$ then establishes $Z \circ (-Y) \in \asym(\mathcal{W})$ and hence $(Z \circ (-Y),-Y) \in \mathcal{A}(N_1,N_2)$.

Furthermore, observe that $S(Z \circ (-Y),-Y) \subseteq S(X,-Y)$ since $Z \in I(X,-Y)$. 
By definition of elimination set we must have at least one coordinate $e$ in $S(X,-Y)$ at which $Z$ has sign $0$,
meaning that $(Z \circ (-Y))_{e} = (-Y)_{e}$ and verifying that the inclusion is proper. This however contradicts minimality.

The same argument also works with the pair $(X,(-Z) \circ X)$ if $Z \in I(-X,Y)$.
Summarising, we have $(I(X,-Y) \cup I(-X,Y)) \cap \mathcal{W} = \emptyset$, which yields $P \in \mathcal{P}(\mathcal{W})$ as desired.
That is, $[P,0] \in \mathcal{W}^\dagger$ fulfills all requirements of $(O4')$ on $V_1 = [N_1,0]$ and $V_2 = [N_2,0]$.\vspace{0.2cm}\\
Case 3.2 of $(O4')$: $N_1,N_2 \in \mathcal{P}(\mathcal{W})$.

Let $N_1 = U + (-U')$ for some $U,U' \in \asym(\mathcal{W})$.
The sign vector $V = (-N_2) \circ U$ then belongs to $\mathcal{W}$ by $(A3)$ and Lemma \ref{Zauber} says that $V \in \asym(\mathcal{W})$.
We use a construction analogous to that in Case 3.1 of $(O4')$ to verify $\mathcal{A}(U,-V) \neq \emptyset$.
Although $(A2)$ no longer guarantees $I_{e}(N_1,N_2) \cap \mathcal{W} \neq \emptyset$ (because $N_1,N_2 \in \mathcal{P}(\mathcal{W})$ implies $N_1,N_2 \notin \mathcal{W}$ ),
one can modify $N_1$ and $N_2$ to make use of Lemma \ref{Zauber}:
By applying $(A3)$ to $N_1,N_2 \in \mathcal{P}(\mathcal{W})$ and $U \in \mathcal{W}$, we infer
that $(\pm N_1) \circ U \in \mathcal{W}$ and $(\pm N_2) \circ U \in \mathcal{W}$.

Pick some $e \in S(N_1,N_2) = S(N_1 \circ U, N_2 \circ U)$.
Axiom $(A2)$ ensures the existence of\vspace{0.2cm}

$A_1 \in I_e(N_1 \circ U,N_2 \circ U) \cap \mathcal{W}$ \text{ } and \text{ } $A_2 \in I_e((-N_1) \circ U,(-N_2) \circ U) \cap \mathcal{W}$.\vspace{0.2cm}\\
Construct $B_1 = A_1 \circ A_2$ and $B_2 = A_2 \circ A_1$ as before.
Since $A_1$ and $A_2$ are members of their respective elimination sets, their supports must be contained in the support of $N_1 \circ U$, which is simply the support of $U$.
Lemma \ref{Zauber} thus ensures $B_1,B_2 \in \asym(\mathcal{W})$.
It remains to check whether $B_1$ and $-B_2$ belong to $I(U,-V)$.
It is readily seen that  $S(N_1,N_2) \subseteq S(N_1 \circ U,N_2 \circ (-U)) = S(U,-V)$.
Conversely, the non-separation set of $U$ and $-V$, i.e. $E - S(U,-V)$, must be contained in the non-separation set of $N_1$ and $N_2$.
The conclusion that $\mathcal{A}(U,-V) \neq \emptyset$ now follows exactly as above with $U$ and $-V$ playing the roles of $X$ and $-Y$.

Now choose $(X,-Y) \in \mathcal{A}(U,-V)$ with minimal separation set $S(X,-Y)$, and let $P = X + (-Y)$.
One obtains $P \in \mathcal{P}(\mathcal{W})$ with the same argument as in Case 3.1 of $(O4')$.
It remains to prove $P \in I(N_1,N_2)$, i.e. $P$ shares the same sign with $N_1$ and $N_2$ at each non-separation coordinate of $N_1$,$N_2$ and the support of $P$ is a proper subset of the support of $N_1$.

Recall that $X,-Y \in I(U,-V) = I(N_1 \circ U,N_2 \circ (-U))$. A non-separation coordinate $f$ of $N_1$ and $N_2$ belongs either to $\underline{N_1}$ or to $N_1^0$.
The case $f \in \underline{N_1}$ is immediately clear as $X,-Y \in I(N_1 \circ U,N_2 \circ (-U))$.
For $f \notin \underline{U}$ we have $P_f = 0$ because $X_f = Y_f =0$, and $(N_1)_f = 0$ because the support of $N_1$ is contained in the support of $U$.
The only remaining case is $f \in \underline{U} - \underline{N_1}$. By Lemma \ref{Zauber} we have $X_f = Y_f = U_f$, whence $P_f = 0 = (N_1)_f$.

Thus it is enough to show that support of $P$ is a proper subset of the support of $N_1$.
Indeed, there is some $e \in S(U,-V) \subseteq \underline{U}$ with $X_e = 0$ by definition of $I(U,-V)$.
Since $\underline{X} = \underline{Y}$, we get $P_e = 0$. However, $(N_1)_e$ can not be 0, for otherwise we would have $X_e = U_e \neq 0$ by Lemma \ref{Zauber} on X, which would contradict $e \notin \underline{X}$. 
\end{proof}

\section{The Flaw in the Proof by Karlander}

The flaw occurs in Case 3.2 of $(O4')$ of the proof of Lemma 3.1: an attempt was made to avoid the statements for $\mathcal{A}(U,-V) \neq \emptyset$ by simply requiring that the pair $(U,-V)$ is put into $\mathcal{A}(U,-V)$.
As demanded by the proof, the set $\mathcal{A}(U,-V)$ should contain those pairs $(X,-Y)$ with two properties, of which the first is $\underline{X} = \underline{Y} \subset \underline{U}$.
This is used to ensure that $X + (-Y) \in I(N_1,N_2)$ in the last step.
The second is that $X + (-Y)$ should be a member of $\mathcal{P(W)}$.
Indeed, in the remainder of the proof of Case 3.2, some pair $(X,-Y) \in \mathcal{A}(U,V)$ eventually provides the sign vector $P = X + (-Y)$ belonging to $\mathcal{P(W)}$.
Unfortunately, the pair $(U,-V)$ itself does not have either of these two properties.
The first fails trivially and the next lemma ensures that the second property fails as well.

\begin{Lem}\label{gegenbsp}
In the context of Case 3.2 of $(O4')$, the pair $(U,V)$ fails to witness that $U + (-V) \in \mathcal{P(W)}$.
\end{Lem}
\begin{proof}
First recall the set-up briefly: $N_1= U + (-U')$ and $N_2$ belong to $\mathcal{P(W)}$. The sign vector $V$ is defined as $(-N_2) \circ U$.

By definition of $\mathcal{P(W)}$, it suffices to show that $I(U,-V) \cap \mathcal{W} \neq \emptyset$.
Note that $N_1' =N_1 \circ U = U$ and $N_2' = N_2 \circ U$ belong to $\mathcal{W}$ by $(A3)$ on $N_2 \in \mathcal{P(W)}$ and $U \in \mathcal{W}$. We first claim $I(N_1',N_2') \subseteq I(U,-V)$.

$Z \in I(N_1',N_2')$ means that the support of $Z$ is a proper subset of the support of $N_1 \circ U$ (which is certainly equal to the support of $U$),
and that $Z$ shares the same sign with $N_1'$ and $N_2'$ at each non-separation coordinate of $N_1'$ and $N_2'$.
It is readily seen that $S(N_1',N_2') = S(N_1 \circ U, N_2 \circ U)$ is a subset of $S(N_1 \circ U, N_2 \circ (-U)) = S(U,-V)$,
or, put differently, that the non-separation set of $U$ and $-V$ is contained in the non-separation set of $N_1'$ and $N_2'$. In particular, $Z$ shares the same sign with $N_1'$ and $N_2'$ at each non-separation coordinate $f$ of $U$ and $-V$. At such a coordinate we necessarily have $(N_1')_f = (N_1)_f = U_f$, implying that $Z \in I(U,-V)$.

By applying $(A2)$ to the pair $N_1'$ and $N_2'$, one derives that $I(N_1',N_2') \cap \mathcal{W} \neq \emptyset$. Hence $I(U,-V) \cap \mathcal{W} \neq \emptyset$ as claimed.
\end{proof}

According to the construction of $\mathcal{A}(U,-V)$ in the original manuscript of Karlander, it would be possible to let the pair $(U,-V)$ play the role of $(X,-Y)$ in the proof of Case 3.2. As Lemma \ref{gegenbsp} demonstrates, there would be no chance of showing that $U + (-V)$ belongs to $\mathcal{P(W)}$.

\section{Generalization of $\mathcal{P}$}
Working with the parallel vector system $\mathcal{P}$ is somewhat involved:
a distinction of cases is complicated because one has to go through all three constrains in the definition of $\mathcal{P}$.
In this section we generalize the sign vector system $\mathcal{P}$ by relaxing the constrains "$X,Y \in \asym(\mathcal{W})$" and "$\underline{X}=\underline{Y}$" in the definition. As equal support is no longer required, we must extend the notion of $I(X,Y)$ accordingly.
\begin{Def}\label{extend}
	Let $E$ be a finite set. For $X,Y \in \{+,-,0\}^E$, we define\vspace{0.2cm}\\
	$\bullet$	the \emph{extended e-elimination set} of $X$ and $Y$ for some $e\in S(X,Y)$ by
	\[I'_e(X,Y) = \{V \in \{+,-,0\}^E : \underline{V} \subseteq (\underline{X} \cup \underline{Y}) - \{e\}, V_f=(X\circ Y)_f  \text{ for all } f \notin S(X,Y)\},\]
	$\bullet$	the \emph{extended elimination set} of $X$ and $Y$ by
	\[I'(X,Y) = \bigcup_{e \in S(X,Y)} I'_e(X,Y).\]
If $S(X,Y)= \emptyset$, then $I'_e(X,Y)=I'(X,Y)=\emptyset$ for every $e\in E$.

For $\mathcal{W} \subseteq \{+,-,0\}^E$ and $X,Y\in \mathcal{W}$, we define\vspace{0.2cm}\\
	$\bullet$	the vector system $\mathcal{Q(W)}$ by	
	\[\mathcal{Q(W)}=\{X + (-Y):X,Y\in \mathcal{W},I'(X,-Y) \cap \mathcal{W}=I'(-X,Y) \cap \mathcal{W}=\emptyset\}.\]
\end{Def}

The set $I'_e(X,Y)$ is adapted to strong elimination, which does not require equal support:
\begin{align*}
(SE)
&\text{ if } X,Y \in \mathcal{W} \text{ and } e \in S(X,Y) \text{ then there exists } Z \in \mathcal{O} \text{ such that}\\
&Z_e=0 \text{ and } Z_f = {(X \circ Y)}_f = {(Y \circ X)}_f \text{ for all } f \notin S(X,Y),
\end{align*}
This axiom is also used to characterize oriented matroids. Actually, the axiomsystem $(O1),(O2),(O3),(SE)$ for oriented matroid is well known.
\begin{Lem}\label{SEO4}
	Assuming $(O3)$, the axioms $(SE)$ and $(O4)$ are equivalent.
\end{Lem}
\begin{proof}
	The forward implication is trivial. As for the backward implication let $X,Y\in \mathcal{W}$ and $e \in S(X,Y)$. By $(O3)$, the pair $X\circ Y$ and $Y \circ X$ are contained in $\mathcal{W}$ with the same support. Obviously, the separation set $S(X,Y)$ is equal to $S(X \circ Y,Y \circ X)$, hence there exists some $Z \in I_e(X\circ Y, Y\circ X) \cap \mathcal{W}$ by $(O4)$, such that $Z_e=0$ and $Z_f = {((X \circ Y)\circ (Y\circ X))}_f = {(X \circ Y)}_f$ for all $f \notin S(X \circ Y,Y \circ X)=S(X,Y)$. Certainly, this $Z$ belongs to $Z \in I_e'(X,Y) \cap \mathcal{W}$, witnessing $(SE)$.
\end{proof}
\begin{Bem}\cite{COM}\label{RSE}
	Note that $X\circ Y \in \mathcal{W}$ entails $(O3)$	 since
	\[X\circ Y = (X \circ -X) \circ Y = X\circ -(X \circ -Y)\in X\circ -\mathcal{W} \subseteq \mathcal{W}.\]
	Therefore, we may replace the axiom $(A1)$ in Theorem \ref{karlander} with
	\[(A1') \text{ if } X,Y \in \mathcal{W} \text{ then } X \circ (-Y) \in \mathcal{W}.\]
	As $(A2)$ only restates $(O4)$ using the notion of elimination set $I$, the axiom $(SE)$ holds in every system $\mathcal{W}$ satisfying $(A1)$ and $(A2)$. Furthermore, one may also rephrase $(SE)$ using the notion of extended elimination set $I'$:
	\[(A2') \text{ if } X,Y \in \mathcal{W} \text{ with } S(X,Y)\neq \emptyset \text{, then } I'_e(X,Y) \cap \mathcal{W} \neq \emptyset \text{ for all $e \in S(X,Y)$}.\]
\end{Bem}
As we have relaxed two constraints, working with $\mathcal{Q}$ is certainly easier than with $\mathcal{P}$. For example, one may rephrase the lemmas and propositions in section 2 and section 3 using $\mathcal{Q}$ instead of $\mathcal{P}$ (with according adjustment). Everything remains true but the number of cases to be considered decreases somewhat. The parallel vector system $\mathcal{P(W)}$ is contained in $\mathcal{Q(W)}$ in general. 

\begin{Kor}\label{subset}
	The inclusion $\mathcal{P(W)} \subseteq \mathcal{Q(W)}$ holds for every $\mathcal{W} \subseteq \{+,-,0\}^E$.
\end{Kor}
\begin{proof}
	Recall the definition of $e$-elimination set $I_e(X,Y)$ for some $X,Y \in \mathcal{W} \subseteq \{+,-,0\}^{E}$ with $\underline{X} = \underline{Y}$:
	\[I_e(X,Y) = \{V \in \{+,-,0\}^E : \underline{V} \subseteq \underline{X} - \{e\} $ and $ V_f=X_f $ for all $ f \notin S(X,Y)\}.\]
	The equations $\underline{X} - \{e\} = (\underline{X} \cup \underline{Y})- \{e\}$ and $X_f = (X\circ Y)_f$ for all $ f \notin S(X,Y)$ hold obviously as $\underline{X} = \underline{Y}$.
	Thus we have $I_e(X,Y) = I'_e(X,Y)$ and consequently $I(X,Y) = I'(X,Y)$ for all $X$ and $Y$ with $\underline{X} = \underline{Y}$.
	Hence for every $X,Y \in \mathcal{W}$ with $P=X+(-Y)\in \mathcal{P(W)}$, the condition $I(X,-Y) \cap \mathcal{W}=I(-X,Y) \cap \mathcal{W}=\emptyset$ is equivalent to $I'(X,-Y) \cap \mathcal{W}=I'(-X,Y) \cap \mathcal{W}=\emptyset$, whence $P \in \mathcal{Q(W)}$. Therefore we have $\mathcal{P(W)} \subseteq \mathcal{Q(W)}.$
\end{proof}
Moreover, if $\mathcal{W}$ is an affine oriented matroid, then $\mathcal{Q(W)}$ is already known.
\begin{Lem}[Bowler]\label{nathan}
	Let $E$ be a finite set. If $\mathcal{W} \subseteq\{+,-,0\}^E$ is an affine oriented matroid, then
	\[\mathcal{Q(W)}=\mathcal{N(W)}.\]
\end{Lem}
\begin{proof}
	
%
	
	First recall that $\mathcal{N(W)} =\{N \in \{+,-,0\}^{E}: (\pm N) \circ \mathcal{W} \subseteq \mathcal{W}\}$. Obviously, the system $\mathcal{Q(W)}$ is symmetric, thus it is sufficient for the forward inclusion to show $\mathcal{Q(W)}\circ \mathcal{W}\subseteq \mathcal{W}$. By Theorem \ref{karlander}, the axioms $(A1),(A2),(A3)$ hold in $\mathcal{W}$, hence $(SE)$ is also fulfilled by Remark \ref{RSE}. In particular, if $I'(X,-Y) \cap \mathcal{W} = \emptyset$ holds, then either $I'(X,-Y) = \emptyset$ (meaning that the separation set $S(X,-Y)$ is empty), or $-Y \notin \mathcal{W}$ (for otherwise $(SE)$ forces a non-empty intersection). Let $Q = X + (-Y)\in \mathcal{Q(W)}$. If $S(X,-Y) = \emptyset$, we note that $Q=X+(-Y)=X\circ (-Y)$, whence $Q\circ \mathcal{W}\subseteq \mathcal{W}$ follows from $(A1)$ because
	\[Q\circ \mathcal{W} = (X+(-Y))\circ \mathcal{W}=X \circ (-Y) \circ \mathcal{W}=X \circ -(Y \circ -\mathcal{W})\subseteq X\circ -\mathcal{W}\subseteq \mathcal{W}.\]
	Hence we may assume that $S(X,-Y)\neq \emptyset$. For using $(A3)$, the same technique as in the proof of Lemma \ref{SEO4} is applied: We work alternatively with the pair $X\circ (-Y)$ and $Y\circ (-X)$ which share the same support $\underline{X} \cup \underline{Y}$. Note that $X+(-Y) = X\circ (-Y) + (-(Y\circ -X)),$ $S(X,-Y)=S(X\circ(-Y),-(Y\circ -X))$ and $I'(X,-Y)=I(X\circ (-Y),-(Y\circ -X))$ as well as $I'(-X,Y)=I((-X)\circ Y,Y\circ (-X))$ hold, as are easy to check by going through the individual coordinates. We conclude as before that $I(X\circ (-Y),-(Y\circ -X))\cap \mathcal{W}=I((-X)\circ Y,Y\circ (-X))\cap \mathcal{W}=\emptyset$ and hence $X\circ (-Y),Y\circ (-X)\in \asym(\mathcal{W})$. Therefore $Q=X+(-Y)= X\circ (-Y) + (-(Y\circ -X))$ belongs to $\mathcal{P(W)}$ and $Q\circ \mathcal{W}\subseteq\mathcal{W}$ follows directly from $(A3)$.
	
	The backward inclusion is trivial by Proposition \ref{N zerlegung} ($\mathcal{N(W)} = \mathcal{P(W)} \cup \sym (\mathcal{W})$): each $N\in \mathcal{P(W)}$ is also a member in $\mathcal{Q(W)}$ as $\mathcal{P(W)}\subseteq \mathcal{Q(W)}$; every $N\in \sym(\mathcal{W})$ can be expressed in $N=N + (-(-N))$, then the pair $X=N,Y=-N$ witnesses that $N=X+(-Y)\in \mathcal{Q(W)}$ since $I'(X,-Y)$ is empty.
\end{proof}

At last we may modify Theorem \ref{karlander} as follows:

\begin{Kor}\label{knauer}
	A set $\mathcal{W} \subseteq\{+,-,0\}^E$ is an affine oriented matroid if and only if $\mathcal{W}$ satisfies $(A1'),(A2')$, and
	\[(A3') \text{ } \mathcal{Q}(\mathcal{W}) \circ \mathcal{W} \subseteq \mathcal{W}.\]
\end{Kor}
\begin{proof}
	Since the axiom $(A2')$ trivially implies $(A2)$, and $(A3)$ is a direct consequence of $(A3')$ by Corollary \ref{subset}, thus $(A1),(A2)$ and $(A3)$ hold in any system $\mathcal{W}$ satisfying $(A1),(A2'),(A3')$. Therefore by Theorem \ref{karlander}, the "if" direction is done. For the "only if" direction let $\mathcal{W}$ be an affine oriented matroid. Combining the result of Theorem \ref{karlander} and Remark \ref{RSE}, we conclude that $\mathcal{W}$ satisfies $(A2')$. It only remains to verify $(A3')$, which is a straightforward consequence of $(A1)$ and $(A3)$ in view of Lemma \ref{nathan} and Proposition \ref{N zerlegung}, respectively: 
	\[\mathcal{Q}(\mathcal{W}) \circ \mathcal{W} \subseteq \mathcal{W}\Longleftrightarrow \mathcal{N(W)\circ W \subseteq W} \Longleftrightarrow \mathcal{\sym(W) \circ W \subseteq W \text{ and } P(W)\circ W \subseteq W}.\]

\end{proof}
\leftline{An alternative proof of Corollary \ref{knauer} has been communicated to us by Knauer \cite{kolja1}.}
\vspace{0.4cm}

The step from oriented matroids to affine oriented matroids could be iterated: just fix one nonzero coordinate and proceed. Since the conditional oriented matroids (COM) can be characterized by $(A1')$ and $(A2')$ \cite{COM}, the iterated process of fixing single coordinates certainly stays within the class of COMs. Is actually every COM obtainable from some oriented matroid in this way? And if so, is this oriented matroid uniquely determinded? For step 1, with affine oriented matroids, the answer is yes in both cases, but it is unclear what happens beyond that first step. Could one at least characterize the COMs that would arise in step 2?

\section*{Acknowledgement.}

We are particularly greatful to Hans-Jürgen Bandelt and Nathan Bowler (Hamburg) for their valuable suggestions and comments that greatly improved the manuscript and motivated the last section.


\begin{thebibliography}{1}
\addtolength{\itemsep}{-1.5ex}

\bibitem{COM}
  H.-J. Bandelt, V. Chepoi, Kolja Knauer,
  \emph{COMs: Complexes of Oriented Matroids}, Preprint. 
\bibitem{bjoerner}
  A. Bj\"orner, M.Las Vergnas, B. Sturmfels, N. White, G.M. Ziegler,
  \emph{Oriented Matroids},
  Cambridge Univ. Press,
  1993.
\bibitem{ck1}
  A. Bj\"orner,
  \emph{Subspace arrangements},
  First European Congress of Mathematics, Vol.1 (Paris 1992), Birkhäuser, Basel, pp. 321-370, 1994.
\bibitem{ck3}
  G. Heinz,
  \emph{The concept of cutting vectors for vector systems in the Euclidean plane},
  Contributions to Algebra and Geometry, 41(1), 181-193, 2000.
\bibitem{karlander}
  J. Karlander,
  A characterization of affine sign vector systems, in \emph{Zero-one matrices matroids and characterization problems}, pp. 67-91, Ph.D. Thesis KTH Stockholm, 1992.
\bibitem{kolja1}
  Kolja Knauer, personal communication.
\bibitem{ck2}
  G. Ziegler,
  \emph{What is a complex matroid?}
  Discrete and Computational Geometry, 10(1), 313-348, 1993.
\bibitem{KTH}
  http://www.emis.de/journals/SLC/wpapers/s34stockholm\_{}des.pdf
\end{thebibliography}
\end{document}